\documentclass[letterpaper, 10 pt, conference]{ieeeconf}  

\usepackage{graphicx}
\usepackage{pgfplots}
\usepackage{subcaption}
\usepackage{cite}
\usepackage{amsmath,amssymb,amsfonts}
\usepackage{mathrsfs}
\usepackage{algorithmicx}
\usepackage{textcomp}
\usepackage{xcolor}
\usepackage{booktabs}
\usepackage{array}
\usepackage[ruled,lined]{algorithm2e}
\usepackage{multirow}
\usepackage{upgreek}
\usepackage{bm}
\usepackage{bbold}
\usepackage{tikz}
\usepackage{tikz-network}
\usepackage{todonotes}

\newtheorem{theorem}{Theorem}

\newtheorem{proposition}{Proposition}

\newtheorem{corollary}{Corollary}
\newtheorem{definition}{Definition}
\newtheorem{example}{Example}

\DeclareMathOperator{\diag}{diag}
\DeclareMathOperator{\proj}{Proj}

\IEEEoverridecommandlockouts                              

\title{\LARGE \bf
A Passivity Analysis for Nonlinear Consensus on Digraphs
}

\author{Feng-Yu Yue and Daniel Zelazo
\thanks{This work was supported by the Israel Science Foundation grant no.\,453/24 and the Gordon Center for Systems Engineering. Feng-Yu Yue and Daniel Zelazo are with the Faculty of Aerospace Engineering, Technion
– Israel Institute of Technology, Haifa 3200003, Israel. Emails: \textsl{fengyu.yue@campus.technion.ac.il}, \textsl{dzelazo@technion.ac.il.}%
}
}

\begin{document}

\maketitle
\thispagestyle{empty}
\pagestyle{empty}

\begin{abstract}
This work presents a passivity-based analysis for the nonlinear output agreement problem in network systems over directed graphs. We reformulate the problem as a convergence analysis on the agreement submanifold. 
First, we establish how passivity properties of individual agents and controllers determine the passivity of their associated system relations.
Building on this, we introduce the concept of submanifold-constrained passivity and develop a novel compensation theorem that ensures output convergence to the agreement submanifold. Unlike previous approaches, our approach can analyze the network system with arbitrary digraphs and any passive agents. We apply this framework to analyze the output agreement problem for network systems consisting of nonlinear and passive agents. Numerical examples support our results.
\end{abstract}

\section{INTRODUCTION}
Multi-agent networks (MANs) have received significant attention in recent years due to their intriguing research challenges and extensive practical applications \cite{sandberg2022secure,amirkhani2022consensus,ChungTRO2018AerialSwarm}. MANs represent the large-scale interconnection of agents, where the analytical complexity arises from both the (nonlinear) dynamics of individual agents and the intricate couplings imposed by the network topology. This work addresses the output agreement problem, one of the fundamental challenges in this fields, where agent outputs are expected to agree on a shared value or trajectory \cite{Mesbahi2010}.

In output agreement analysis, it is essential to examine the individual agent dynamics and the network topology governing their interactions separately to mitigate the inherent complexity of MANs. Here, passivity theory plays a significant role. The seminal work by \cite{arcak2007passivitydesign} introduced passivity as a key tool for characterizing group coordination behavior. Then, in \cite{Burger2013duality, Sharf2019MIMO}, a comprehensive passivity-based cooperative control framework was developed. The framework established a connection between consensus behaviors and dual network optimization problems \cite{Rockafellar1998NetOpt}. Montenbruck et al. \cite{Max2017submanifold} regarded the agreement space as a submanifold and developed powerful analytical tools to establish connections between passivity properties of relations and submanifold stabilization problems (including consensus problems).

However, the research mentioned shares a common implicit assumption: that the underlying graphs of MANs are \emph{undirected}. This can sometimes be restrictive and may lead to inefficiencies, so it is necessary to consider network systems that interact over directed graphs (digraphs). In general, analyzing these systems presents challenges, even in the linear case \cite{OlfatiSaber2007}. On the other hand, although passivity theory plays a significant role in simplifying the analysis of MANs, its application to network systems with digraphs remains underexplored. While the work \cite{li2019passivition} introduced a passivation approach tailored for the consensus problem of directed networks, their scope was confined to scenarios involving linear static controllers.


In \cite{yue2024balance_passivity}, we proposed a framework enabling passivity-based analysis for directed networks and reformulated the output consensus problem as convergence analysis on a submanifold. Our approach includes two steps: First, we defined input-output relations for the forward and feedback paths with respect to the agreement submanifold, and derived passivity-based inequalities for these relations.  Second, we established a compensation theorem that leveraged excess passivity in the feedback path to compensate for the shortage in the forward path, thereby guaranteeing output convergence to the submanifold.

Despite these contributions, the analysis in \cite{yue2024balance_passivity} was restricted to network systems with balanced digraphs and output-strictly passive agents and controllers, addressing only the stabilization problem. Additionally, we will discuss in subsection \ref{ssec:limitations_ecc} that the compensation theorem in \cite{yue2024balance_passivity} becomes inapplicable when agents exhibit only passivity or input-strict passivity. In these scenarios, the forward path relation exhibits a passivity shortage that cannot be compensated by the feedback path's excess passivity. These unresolved limitations constitute the primary motivation for the present work.

This paper focuses on the reformulated output agreement problem for network systems over digraphs. We extend the two-step approach from \cite{yue2024balance_passivity} to arbitrary directed graphs with passive agents. Our contributions are as follows. First, we establish how passivity properties of individual agent and controller determine the passivity of their associated system relations. This reveals fundamental limitations of the existing compensation theorem in \cite{yue2024balance_passivity}. To overcome these challenges, we introduce the notion of submanifold-constrained passivity and propose a novel compensation theorem that guarantees output convergence to the agreement submanifold. This new approach eliminates the need to construct a storage function for each individual agent, requiring only a constrained storage function for the entire forward path. Finally, we apply our framework to integrator-like nonlinear agents and derive sufficient conditions for directed networks to achieve output agreement.

The remainder of this paper is organized as follows. Section \ref{sec:preliminary} presents the necessary background on network systems over digraphs, passivity, and the reformulated output agreement problem. Section \ref{sec:submanifold_passivity} introduces submanifold-constrained passivity and develops a generalized compensation theorem for output agreement. Section \ref{sec:passive_agreement} applies this theorem to analyze output agreement in a specific class of network systems. Numerical simulations and conclusions are presented in Sections~\ref{sec:case_study} and~\ref{sec:concluding}, respectively. Due to space constraints, this paper presents only selected proofs. The complete proofs can be found in \cite{}.

\paragraph*{Notations}
We denote the $n$-dimensional all-one vector and identity matrix by $\mathbb{1}_n$ and $I_n$, respectively. For a set $A$, $\vert A\vert$ denotes its cardinality. The orthogonal complement of a subspace $U$ is denoted by $U^\perp$, and $\proj_U(x)$ represents the orthogonal projection of $x \in \mathbb{R}^n$ onto $U$.

A digraph $\mathcal{D}=(\mathcal{V}, \mathcal{E})$ consists of a finite vertex set $\mathcal{V}$ and edge set $\mathcal{E}\subset\mathcal{V}\times \mathcal{V}$, with undirected counterpart $\mathbb{G}$. A digraph contains a globally reachable node if one node can be reached from all others via directed paths \cite{bullo2020lectures}. The incidence matrix $E\in \mathbb{R}^{|\mathcal{V}|\times|\mathcal{E}|}$ is defined as $[E]_{ik}=1$ if $i$ is the head of edge $e_k$, $[E]_{ik}=-1$ if $i$ is the tail, and $[E]_{ik}=0$ otherwise. We decompose $E=B_o+B_i$ into out-incidence matrix $B_o$ (where $[B_o]_{ik}:=1$ if $i$ is the head of edge $e_k=(i,k)$ and $[B_o]_{ik}:=0$ otherwise) and in-incidence matrix $B_i$ (where $[B_i]_{ik}:=-1$ if $i$ is the tail of edge $e_k$ and $[B_o]_{ik}:=0$ otherwise) \cite{restrepo2021edgeLyapunov}. The graph Laplacian, in-Laplacian and out-Laplacian matrices are defined as $L(\mathbb{G})=EE^\top$, $L_i(\mathcal{D})=B_iE^\top$, and $L_o(\mathcal{D})=B_oE^\top$, respectively.

\section{PRELIMINARIES}\label{sec:preliminary}
This section presents the necessary preliminaries on networked systems over directed graphs and passivity theory, and then establishes the connection between output agreement and convergence to the agreement submanifold.
\subsection{Network systems over digraphs}
Consider a population of agents interconnected over a directed graph $\mathcal{D}=(\mathcal{V},\mathcal{E})$, where $\mathcal{V}$ represents the set of agents and $\mathcal{E}$ indicates existence of dynamical coupling between agents (the edge controllers). 
Each agent $\{\Sigma_i\}_{i\in\mathcal{V}}$ and controller $\{\Pi_k\}_{k\in\mathcal{E}}$ are described by SISO nonlinear dynamical systems,
\begin{align}
&\hspace{-0.1em}\Sigma_i:\dot{x}_i(t)=f_i(x_i(t),u_i(t)), \
        y_i(t)=h_i(x_i(t),u_i(t)), \label{eq:agents_decomp}\\
        &\hspace{-0.1em}\Pi_k:\hspace{-0.1em}\dot{\eta}_k(t)=\phi_k(\eta_k(t),\hspace{-0.1em}\zeta_k(t)),
        \hspace{-0.1em}\mu_k(t)=\psi_k(\eta_k(t),\hspace{-0.1em}\zeta_k(t)).\label{eq:controllers_decomp}
\end{align}
In network systems over digraphs, agents and controllers are interconnected in a parallel configuration. With a slight abuse of notation, we denote the parallel interconnection of agents as $\Sigma=\diag(\Sigma_1, \ldots,\Sigma_{|\mathcal{V}|})$ and that of controllers as $\Pi=\diag(\Pi_1, \ldots,\Pi_{|\mathcal{E}|})$. Accordingly, we define $u(t)=[u_1, \ldots, u_{|\mathcal{V}|}]^\top$ as the stacked input vector of all agents, with analogous definitions for the agent outputs $y(t)$, controller inputs $\zeta(t)$, and controller outputs $\mu(t)$. 

Fig. \ref{fig:dnetworks} illustrates two block diagrams for network systems over digraphs. These structures are equivalent. However, as shown in \cite{yue2024balance_passivity}, the structure $(\Sigma,\Pi,\mathcal{D})$ in Fig. \ref{fig:dnet}, is not ideal for passivity-based analysis due to the potential loss of passivity in the feedback path, even with passive controllers. To address this, we introduced the structure $(\Sigma,\Pi,\mathcal{D},w)$ depicted in Fig. \ref{fig:decomp}. In this structure, the symmetric interconnection operator $E\Pi E^\top$ in the feedback path remains passive \cite{arcak2007passivitydesign}, provided the controllers, $\Pi$, are passive. While the passivity of the overall system cannot be guaranteed, we can treat $w=B_i\mu$ as an external output and perform passivity-based analysis on the feedback interconnection with the input-output pair $(w,y)$.

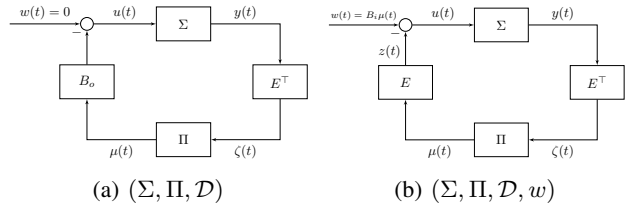
\begin{figure}[ht]
\centering
\tikzstyle{block} = [draw, rectangle, 
    minimum height=2.5em, minimum width=4em]
\tikzstyle{sum} = [draw, circle, node distance=1cm]
\tikzstyle{input} = [coordinate]
\tikzstyle{output} = [coordinate]
\tikzstyle{pinstyle} = [pin edge={to-,thin,black}]
\begin{subfigure}{0.23\textwidth}
    \resizebox{\textwidth}{!}{
\begin{tikzpicture}[scale=0.8,auto, node distance=2.5cm,>=latex']
   
    \node [input, name=input] {};
    \node [sum, right of=input, node distance=2cm] (sum) {};
    \node [block, below of=sum,
            node distance=1.5cm] (Eout) {
                $B_o$};
    \node [block, right of=sum] (agents) {$\Sigma$};
    \node [input, name=center, below of=agents] {};
    \node [block, below of=agents,
            node distance=3cm] (controllers) {$\Pi$};
    \node [output, right of=agents] (output) {};
    \node [block, below of=output,
            node distance=1.5cm] (ETout) {$E^\top$};

    \draw [draw,->] (input) -- node {$w(t)=0$} (sum);
    \draw [->] (sum) -- node {$u(t)$} (agents);
    \draw [->] (agents) -| node [near start] {$y(t)$}  (ETout);
    \draw [->] (ETout) |- node [near end] {$\zeta(t)$}  (controllers);
    \draw [->] (controllers) -| node [near start] {$\mu(t)$}  (Eout);
    \draw [->] (Eout) --  node[pos=0.95] {$-$}  (sum);
\end{tikzpicture} 
}   
\caption{$(\Sigma,\Pi,\mathcal{D})$}
    \label{fig:dnet}
\end{subfigure}
\begin{subfigure}{0.23\textwidth}
    \resizebox{\textwidth}{!}{
\begin{tikzpicture}[scale=0.8, auto, node distance=2.5cm,>=latex']
    \node [input, name=input] {};
    \node [sum, right of=input, node distance=2cm] (sum) {};
    \node [block, below of=sum,
            node distance=1.5cm] (Eout) {
                $E$};
    \node [block, right of=sum] (agents) {$\Sigma$};
    \node [input, name=center, below of=agents] {};
    \node [block, below of=agents,
            node distance=3cm] (controllers) {$\Pi$};
    \node [output, right of=agents] (output) {};
    \node [block, below of=output,
            node distance=1.5cm] (ETout) {$E^\top$};

    \draw [draw,->] (input) -- node {\scriptsize{$w(t)=B_i\mu(t)$}} (sum);
    \draw [->] (sum) -- node {$u(t)$} (agents);
    \draw [->] (agents) -| node [near start] {$y(t)$}  (ETout);
    \draw [->] (ETout) |- node [near end] {$\zeta(t)$}  (controllers);
    \draw [->] (controllers) -| node [near start] {$\mu(t)$}  (Eout);
    \draw [->] (Eout) -- node[pos=0.4]{$z(t)$} node[pos=0.95] {$-$}  (sum);
\end{tikzpicture} 
}   
\centering
\caption{$(\Sigma,\Pi,\mathcal{D},w)$}
    \label{fig:decomp}
\end{subfigure}
\caption{Two structures of network systems over directed graphs.}
\label{fig:dnetworks}
\end{figure}
 
In this context, the feedback equations can be described from an input-output perspective by defining $w(t)=B_i\mu(t)$ and $z(t)=E\mu(t)$, which leads to
\begin{equation}\label{eq:net_relations}
\begin{array}{cc}
   \hspace{-.2cm} u(t)=w(t)-z(t)=-B_o\mu(t), \text{ and } \zeta(t)=E^\top y(t).
\end{array}  
\end{equation}
In this work, we focus solely on the structure $(\Sigma,\Pi,\mathcal{D},w)$.

\subsection{Passivity}
We take the system (\ref{eq:agents_decomp}) as an example to introduce passivity.
\begin{definition}[\cite{Khalil2008Nonlinear,Burger2013duality, sharf2020geometricpassivation}]\label{def:passivity}
    For the SISO system \eqref{eq:agents_decomp}, if there exists a positive semi-definite storage function $V_i(x_i)$, and scalars $\varepsilon_i$, $\delta_i$ such that
    \begin{equation}
        \dot{V_i}(x_i)\leq u_iy_i-\varepsilon_iy_i^2-\delta_i u_i^2, \quad \forall x_i,u_i,y_i,
    \end{equation}
    then, the system \eqref{eq:agents_decomp} is said to be
    \begin{enumerate}
        \item \emph{passive} if $\varepsilon_i=0$ and $\delta_i=0$,
        \item \emph{output strictly passive} (OP-$\varepsilon_i$)  if $\varepsilon_i>0$ and $\delta_i\geq0$,
        \item \emph{input strictly passive} (IP-$\delta_i$) if $\varepsilon_i\geq0$ and $\delta_i>0$.
        \item \emph{input-output passive} (IOP-$\delta_i,\varepsilon_i$) if $\varepsilon_i,\delta_i\in\mathbb{R}$ and $\varepsilon_i \delta_i < \tfrac{1}{4}$ (see more details in \cite{sharf2020geometricpassivation}). \hfill $\triangledown$
    \end{enumerate}
\end{definition}

\subsection{Output agreement and convergence to the agreement submanifold}\label{ssec:agreement_prelation}
This work analyzes the output agreement problem for network systems interacting over directed graphs from a passivity-based perspective. We first transform this problem to study the convergence of outputs to a submanifold and introduce how to conduct a passivity-based analysis for the reformulated problem.

\begin{definition}
   Consider the network system $(\Sigma,\Pi,\mathcal{D},w)$ defined in \eqref{eq:agents_decomp}-\eqref{eq:net_relations}. The system is said to achieve \emph{asymptotic output agreement} if for all initial conditions, 
$$\lim_{t\to\infty}y(t)=c\mathbb{1}_{\vert \mathcal{V}\vert}\in\mathrm{span}(\mathbb{1}_{\vert \mathcal{V}\vert}),$$
    where $c\in\mathbb{R}$ is the agreement value.  \hfill $\triangledown$
\end{definition}

Since $\mathrm{span}(\mathbb{1})$ forms a linear submanifold \cite{Max2017submanifold}, we denote it as the \emph{agreement submanifold} $S$, with orthogonal complement $S^\perp$ termed the \emph{disagreement submanifold}. Output consensus thus corresponds to convergence to $S$.
To determine convergence of $y(t)$ to $S$, we now introduce the space $\mathscr{L}_S^p$ defined with respect to $S$ as (\cite{Max2017submanifold,yue2024balance_passivity}),
\begin{equation}\label{eq:space_LSp}
    \mathscr{L}_S^p=\left\{f:\mathbb{R}\to U\,| {\mathsmaller{\int}}_\mathbb{R}d(f(t),S)^p {\rm d}t<\infty\right\},
\end{equation}
where $1\leq p<\infty$, $d(x,M)$ denotes the infimal Euclidean distance from $x$ to all the points in an embedded submanifold $M\subseteq \mathbb{R}^n$, and $U$ is a tubular neighborhood of $S$ \cite[Definition 1]{Max2017submanifold}.
The existence of $U$ is guaranteed by the fact that $S$ is smoothly embedded \cite[Chapter 10]{lee2024smoothmanifolds}. 
To handle outputs from unstable systems, we define the truncation operator that is also the orthogonal projection onto $S$:
$$r:\mathbb{R}^n\to S, \ x\mapsto\proj_S(x)=\tfrac{1}{n}\mathbb{1}_n\mathbb{1}_n^\top x,$$
and introduce the extended space
$$\bar{\mathscr{L}}_S^p= \left\{f|f_S^\tau\in\mathscr{L}_S^p, \  \forall \tau\in[0,\infty]\right\},$$
where $f_S^\tau(t)= f(t)$ for $t\in[0,\tau)$ and $f_S^\tau(t)= \proj_S(f(t))$ otherwise. For passivity-based analysis, we set $p=2$. However, ${\mathscr{L}}_S^p$ is not a normed space \cite{Max2017submanifold} and it is impossible to define passivity within ${\mathscr{L}}_S^p$ as we have no inner product on this space. Thus, we introduce the mapping

\begin{small}
\begin{equation}\label{eq:ThetaS}
     \Theta_S: \bar{\mathscr{L}}_S^2\to \bar{\mathscr{L}}^2, \ f(t)\mapsto \proj_{S^\perp}(f(t))=\left(I_n-\tfrac{1}{n}\mathbb{1}_n\mathbb{1}_n^\top\right)f(t),
\end{equation}
\end{small}
which projects signals onto
$S^\perp$.


Next, to connect output consensus to passivity, we define input-output relations for the forward path (agent relation $H_a$) and feedback path (controller relation $H_c$) of network system $(\Sigma,\Pi,\mathcal{D},w)$:
\begin{align}
    &(u(t),y(t))\in H_a\subset (\bar{\mathscr{L}}^2 \times \bar{\mathscr{L}}_S^2), \label{eq:Arelaiton}\\
    &(y(t),z(t))\in H_c\subset (\bar{\mathscr{L}}_S^2 \times \bar{\mathscr{L}}^2), \label{eq:Crelation}
\end{align}
where $u(t)+z(t)=w(t)\in\bar{\mathscr{L}}^2$.

Using these notions, the inner product between $u(t)\in\bar{\mathscr{L}}^2$ and $y(t)\in\bar{\mathscr{L}}_S^2$ is given by $$\langle u^\tau,\Theta_S(y_S^\tau)\rangle={\mathsmaller{\int}}_{-\infty}^\tau u^\top(t) \proj_{S^\perp}(y(t)) {\rm d}t,$$
which corresponds to a special case of \cite[Equation 51]{Max2017submanifold}.
This enables establishing passivity-based inequalities for $H_a$ and $H_c$ in terms of $u^\top \proj_{S^\perp}(y)$ and $z^\top \proj_{S^\perp}(y)$, respectively. Furthermore, if $\lim\limits_{t\to\infty}\proj_{S^\perp}(y(t))=0$ for bounded $y(t)$, then network system $(\Sigma,\Pi,\mathcal{D},w)$ achieves output agreement.

In \cite{yue2024balance_passivity}, we derived sufficient conditions for output agreement in network systems over balanced digraphs with output-strictly passive agents and edge controllers in two steps. First, establish passivity-based inequalities for $H_a$ and $H_c$ (using the passivity properties of agents and controllers):
\begin{align}
    &\dot{V}(x) \leq u^\top(t) \proj_{S^\perp}(y(t)) + g_1(u(t), y(t)), \label{eq:diss_agent} \\
    &\dot{W}(\eta) \leq z^\top(t) \proj_{S^\perp}(y(t)) + g_2(\mu(t), \zeta(t)), \label{eq:diss_controller}
\end{align}
where $g_1, g_2: \mathbb{R}^{\vert \mathcal{V}\vert} \times \mathbb{R}^{\vert \mathcal{V}\vert} \to \mathbb{R}$, and $V(x)$, $W(\eta)$ are positive semi-definite storage functions. Second, apply the following compensation theorem to analyze output agreement for $(\Sigma,\Pi,\mathcal{D},w)$.


\begin{theorem}[\cite{yue2024balance_passivity}]\label{thm:passivity_diss_outconsensus}
Consider a network system $(\Sigma, \Pi, \mathcal{D}, w)$, where $\mathcal{D}$ is a directed and with a globally reachable node. For agent dynamics \eqref{eq:agents_decomp}, let $y_i(t) = h_i(x_i)$ and assume that $f_i$ and $h_i$ are continuously differentiable. For the agent relation $H_a$ with \eqref{eq:diss_agent}, if there exists a controller relation $H_c$ with \eqref{eq:diss_controller} and a positive constant $\varepsilon$, such that the sum of \eqref{eq:diss_agent} and \eqref{eq:diss_controller} satisfies
\begin{equation}\label{eq:w_projy_QW}
    w^\top \proj_{S^\perp}(y)\geq \dot{V}+\dot{W}+\varepsilon\|\proj_{S^\perp}(y)\|_2^2,
\end{equation}
then, $\lim\limits_{t\to\infty}\hspace{-0.1cm}\proj_{S^\perp}(y(t))\hspace{-0.08cm}=\hspace{-0.08cm}0$, achieving output agreement. \hfill $\triangledown$
\end{theorem}

This theorem is a direct summary and extension of Propositions 4 and 5, and Theorem 1 in \cite{yue2024balance_passivity}, where inequality \eqref{eq:w_projy_QW} corresponds to inequality (30) in \cite{yue2024balance_passivity}.

However, deriving passivity-based inequalities for the agent relation directly from passivity properties of agents presents certain limitations. In Section \ref{sec:submanifold_passivity}, we will discuss these limitations and propose a potential solution.

For brevity, we will henceforth omit the time argument $t$ of signals in pointwise passivity-based inequalities. Instead, we will use the shorthand notation $u,y,\zeta$ and $\mu$.

\section{Output agreement on digraphs and submanifold-constrained passivity}\label{sec:submanifold_passivity}
This section focuses on the output agreement problem for network systems $(\Sigma^o,\Pi,\mathcal{D},w)$ where $\mathcal{D}=(\mathcal{V},\mathcal{E})$ and agents $\Sigma_i^o$ follow the dynamics,
\begin{equation}\label{eq:sigma_o}
        \Sigma_i^o:\quad \dot{x}_i(t)=f_i(x_i(t),u_i(t)),\ \ 
        y_i(t)=h_i(x_i(t)), \ \ i\in\mathcal{V}
\end{equation}
where $f_i:\mathbb{R}^{n_i}\times\mathbb{R}\to\mathbb{R}^{n_i}$ and $h_i:\mathbb{R}^{n_i}\to\mathbb{R}$ are continuously differentiable functions.

We begin by examining the limitations of the existing approach from \cite{yue2024balance_passivity} and then generalize the approach to analyze network systems with arbitrary digraphs and any passive agents. This extension is facilitated by the concept of submanifold-constrained passivity for output agreement.

\subsection{From individual passivity to characterizations of relations: The limitations}\label{ssec:limitations_ecc}
This subsection reveals fundamental limitations of directly deriving passivity relations from individual passivity properties, as is done in the first step of the existing method.

Consider network systems over digraphs $(\Sigma^o, \Pi,\mathcal{D},w)$. To ensure the generality of our analysis, we assume the agents and controllers are input-output passive (see Definition \ref{def:passivity}), as input-output passivity encompasses all cases of passivity. The following proposition answers how agents' passivity affects the passivity-based properties of agent relation $H_a$.

\begin{proposition}\label{prop:IOP_agents}
    Consider a group of $|\mathcal{V}|$ SISO agents \eqref{eq:sigma_o}. Assume that for all $i\in\{1,\ldots,|\mathcal{V}|\}$, the agents $\Sigma_i^o$ are IOP-($\delta_i,\varepsilon_i$) where $\delta_i\varepsilon_i< \tfrac{1}{4}$. Let
    $\delta=\min_i(\delta_i)$ and $\varepsilon=\min_i(\varepsilon_i)$. Then, it follows that 
    \begin{equation}\label{eq:IOP_u_projy}
u^\top \proj_{S^\perp}(y)\geq \sum\nolimits_{i=1}^{|\mathcal{V}|} \dot{V}_i+(\delta-\tfrac{1}{2})\|u\|_2^2+(\varepsilon-\tfrac{1}{2})\|y\|_2^2,
\end{equation}
where $V_i(x)$'s are positive semi-definite storage functions. $\triangledown$
\end{proposition}

Next, derive the characterization of the controller relation $H_c$ from the controllers' passivity.
\begin{proposition}\label{prop:IOP_controllers}
    Consider a group of $|\mathcal{E}|$ SISO edge controllers \eqref{eq:controllers_decomp} in system $(\Sigma^o,\Pi,\mathcal{D})$. Let $\lambda_2$ be defined as above. Assume that for all $k\in\{1,\ldots,|\mathcal{E}|\}$, the controllers $\Pi_k$ are IOP-($\gamma_k,\alpha_k$) where $\gamma_k,\alpha_k\geq 0$ and $\gamma_k\alpha_k<\tfrac{1}{4}$ for all $k$. Then, it follows that
    \begin{equation}\label{eq:IOP_z_projy}
        z^\top \proj_{S^\perp}(y)\geq \sum\nolimits_{k=1}^{|\mathcal{E}|} \dot{W}_k+\alpha\|\mu\|_2^2+\gamma\lambda_2\|\proj_{S^\perp}(y)\|_2^2,
    \end{equation} 
    where $W_k(\eta)$'s denote positive semi-definite storage functions, $\alpha=\min_k(\alpha_k)$ and $\gamma=\min_k(\gamma_k)$. \hfill $\triangledown$
\end{proposition}
Note that this proposition holds for all digraphs. While this proposition is stated for non-negative $\alpha_k$ and $\gamma_k$, it remains valid for negative values if we substitute $\lambda_{|\mathcal{V}|}$ for $\lambda_2$ in \eqref{eq:IOP_z_projy}.

Now, consider the case where $\delta-\tfrac{1}{2}<0$ and $\varepsilon-\tfrac{1}{2}<0$ in \eqref{eq:IOP_u_projy}. To apply Theorem \ref{thm:passivity_diss_outconsensus}, \eqref{eq:IOP_z_projy} must compensate for these negative terms in \eqref{eq:IOP_u_projy}. While an appropriate $\alpha$ can compensate for $(\delta-\tfrac{1}{2})\|u\|_2^2$ \cite{yue2024balance_passivity}, compensation for $(\varepsilon-\tfrac{1}{2})\|y\|_2^2$ is impossible since \eqref{eq:IOP_z_projy} lacks a $\|y\|_2^2$ term. Consequently, the method in \cite{yue2024balance_passivity} cannot handle passive agents ($\delta=\varepsilon=0$) or input-strictly passive agents. 

This limitation is particularly significant since linear consensus protocols contains passive integrator agents and output-strictly passive controllers \cite{Mesbahi2010,Khalil2008Nonlinear}, so the existing approach is inadequate for analyzing even basic linear consensus protocols. To address this limitation, we introduce the concept of submanifold-constrained passivity.

\subsection{Submanifold-constrained passivity for output agreement}
Before defining submanifold-constrained passivity for the output agreement problem, we introduce the notion of a \emph{submanifold-constrained storage function}. Recall that storage functions $V(x)$ are continuously differentiable and positive semi-definite with $V(x)=0$ when $x=0$. They are inadequate for submanifold analysis as they cannot distinguish whether $x$ lies within or outside an expected submanifold. Inspired by \cite{Max2017submanifold}, we present the following refined definition to address this limitation.

\begin{definition}[$M$-constrained storage function]\label{def:SPassivity}
    Consider a smoothly embedded submanifold $M$. Let $x\in\mathbb{R}^n$ and $f:\mathbb{R}^n \to \mathbb{R}^n$. A differentiable function $Q: \mathbb{R}^n\to\mathbb{R}$ is called \emph{a constrained storage function} with respect to $M$, if
    \begin{enumerate}
        \item $Q(x)=0$, for all $x$ with $f(x)\in M$, and,
        \item $Q(x)>0$, otherwise.  \hfill $\triangledown$
    \end{enumerate}
\end{definition}

The submanifold-constrained passivity for the agreement problem is now defined.
\begin{definition}
    Consider a group of $n$ SISO systems $\Sigma_i^o$ and the agreement submanifold $S$. Suppose the systems are interconnected in parallel, and denote them by $\Sigma^o$. If there exists an $S$-constrained storage function $Q(x)$ and numbers $\delta,\varepsilon\in\mathbb{R}$, such that for all $t$,
    \begin{equation}\label{eq:SPassivity}
        u^\top \proj_{S^\perp}(y)\geq \tfrac{\partial Q(x)}{\partial x}\dot{x}+\varepsilon\|\proj_{S^\perp}(y)\|_2^2+\delta\|u\|_2^2,
    \end{equation}
    then, the system $\Sigma^o$ is said to be
    \begin{enumerate}
        \item \emph{$S$-passive} if $\varepsilon=\delta=0$,
        \item \emph{input $S$-passive} if $\varepsilon=0$ and $\delta\neq0$,
        \item \emph{output $S$-passive} if $\varepsilon\neq0$ and $\delta=0$,
        \item \emph{input-output $S$-passive} if $\varepsilon\delta<\tfrac{1}{4}$.  \hfill $\triangledown$
    \end{enumerate}
\end{definition}

By introducing submanifold-constrained passivity for output agreement, if an $S$-constrained storage function exists, the passivity-based inequality \eqref{eq:SPassivity} of agent relations $H_a$ no longer includes the $\|y\|_2^2$ term. This framework eliminates the need for individual storage functions for each agent, requiring only a single constrained storage function for the agent relation, thereby simplifying analysis and design.
\begin{example}
Consider $n$ SISO agents interconnected in parallel with dynamics $\Sigma: \dot{x}=x+u, \ y=x$. Then $\Sigma$ is output $S$-passive with $Q(x)=\tfrac{1}{2}x^\top(I-\tfrac{1}{n}\mathbb{1}\mathbb{1}^\top)x$.

We first show that $Q(x)$ is an $S$-constrained storage function. Since $(I-\tfrac{1}{n}\mathbb{1}\mathbb{1}^\top)$ is symmetric with eigenvalues $\{0,1\}$, the Rayleigh quotient yields $0\leq Q(x)\leq x^\top x$ \cite[Theorem 4.2.2]{horn2012matrix}. The kernel of $(I-\tfrac{1}{n}\mathbb{1}\mathbb{1}^\top)$ is $S=\mathrm{span}(\mathbb{1}_n)$, so, with $f(x)=x$, we have $Q(x)=0$ if and only if $x\in S$, and $Q(x)>0$ otherwise.
Then, computing $\tfrac{\partial Q}{\partial x}\dot{x}=x^\top(I-\tfrac{1}{n}\mathbb{1}\mathbb{1}^\top)(x+u)=u^\top\proj_{S^\perp}(y)+\|\proj_{S^\perp}(y)\|_2^2$ confirms output $S$-passivity.
\hfill $\triangledown$
\end{example}




Now, with the notion of submanifold-constrained passivity, we can establish a new compensation theorem.
\begin{theorem}\label{thm:SPassivity_outconsensus}
Consider a network system $(\Sigma^o, \Pi, \mathcal{D}, w)$, where $\mathcal{D}$ is a directed graph containing a globally reachable node. Suppose an $S$-constrained storage function $Q(x)$ exists so that the agent relation $H_a$ satisfies \eqref{eq:SPassivity}. If there exists a controller relation $H_c$, characterized by \eqref{eq:diss_controller} with a storage function $W(\eta)$, and a positive constant $\varepsilon$ such that the sum of \eqref{eq:SPassivity} and \eqref{eq:diss_controller} satisfies
\begin{equation}\label{eq:passive_compensate}
    w^\top \proj_{S^\perp}(y)\geq \dot{Q}+\dot{W}+\varepsilon\|\proj_{S^\perp}(y)\|_2^2,
\end{equation}
then the network system achieves output agreement. \hfill $\triangledown$
\end{theorem}

Comparing \eqref{eq:passive_compensate} with \eqref{eq:w_projy_QW} in Theorem~\ref{thm:passivity_diss_outconsensus}, the key distinction is the utilization of the $S$-constrained storage function in \eqref{eq:passive_compensate}. Note that this compensation theorem can be applied to the study of network systems with arbitrary digraphs and agents exhibiting any passive properties.

This theorem emphasizes that constructing a constrained storage function is crucial for enabling a passivity-based analysis of the output agreement problem. While the theorem is somewhat abstract, we will apply it to investigate the output agreement of a group of passive agents and characterize the form of the $S$-constrained storage function.

\section{Output agreement of passive agents}\label{sec:passive_agreement}
This section analyzes the output agreement problems for network systems composed of a special class of passive systems, as described by,
\begin{equation}\label{eq:affine_upsilon}
    \Upsilon_i:\ \dot{x}_i(t)=u_i(t), \ \
    y_i(t)=h_i(x_i(t)), \ \ i\in\mathcal{V},
\end{equation}
where $x_i(t), u_i(t), y_i(t) \in \mathbb{R}$, and $h_i: \mathbb{R}^{n_i} \to \mathbb{R}$ are continuously differentiable monotone passive functions. We exclude the trivial case where $h_i(x_i) \equiv 0$. Note that this system is integrator-like, representing a cascade connection of an integrator and a passive memoryless function $h_i$.

Consider now the network system $(\Upsilon,\Pi,\mathcal{D},w)$. Agents $\Upsilon$ are passive, so the method proposed in \cite{yue2024balance_passivity} cannot be applied to analyze the output agreement problem for these systems, as discussed in subsection \ref{ssec:limitations_ecc}. To address this, we construct an $S$-constrained storage function for the agent relation, and provide sufficient conditions under which the systems achieve output agreement.

The following proposition provides sufficient conditions for the existence of such constrained storage functions.
\begin{proposition}\label{prop:dissipativity}
    Consider a group of $\vert\mathcal{V}\vert$ SISO agents $\Upsilon_i$. Let the agent relation $H_a$ be defined as above and $P=I-\tfrac{1}{\vert\mathcal{V}\vert}\mathbb{1}\mathbb{1}^\top$. Assume that for all $i\in\{1,\ldots,\vert\mathcal{V\vert}\}$ and $s\in\mathbb{R}$, there exists a positive constant $m$ such that $\tfrac{{\mathrm d}h_i(s)}{{\mathrm d}s}\leq m$, then, $H_a$ exhibits the following passivity-based characterization,
\begin{equation}\label{eq:upsilon_dissipativity}
    \begin{aligned}
         u^\top \proj_{S^\perp}(y)\geq \dot{Q}(x)-\tfrac{M}{2}\|u\|_2^2-\tfrac{M}{2}\|\proj_{S^\perp}(y)\|_2^2,
    \end{aligned}
    \end{equation}
    with the $S$-constrained storage function, $$Q(x)=\tfrac{1}{2}h^\top(x)P h(x),$$ where $M=\max({1,|1-m|})$. \hfill $\triangledown$
\end{proposition}
\begin{proof}
    We start by computing the derivative of $Q(x)$ with respect to $t$, i.e.,
    \begin{equation}
        \dot{Q}(x)=\tfrac{\partial Q(x)}{\partial x}\dot{x}=\tfrac{\partial Q(x)}{\partial h(x)}\tfrac{\partial h(x)}{\partial x}\dot{x}=(Ph(x))^\top\tfrac{\partial h(x)}{\partial x} u,
    \end{equation}
    where $\tfrac{\partial h(x)}{\partial x}=\diag\left(\left[\tfrac{\partial h_1(x_1)}{\partial x_1},\ldots,\tfrac{\partial h_{\vert\mathcal{V}\vert}(x_{\vert\mathcal{V}\vert})}{\partial x_{\vert\mathcal{V}\vert}}\right]\right)$ is the Jacobian of $h(x)$.
    
    Now, considering $u^\top \proj_{S^\perp}(y)-\dot{Q}(x)$, we arrive that,
    \begin{equation}\label{eq:uprojy_dotQ}
    \begin{aligned}
        &u^\top \proj_{S^\perp}(y)-\dot{Q}(x)=u^\top (I-\tfrac{\partial h(x)}{\partial x})\proj_{S^\perp}(y)\\
        &=u^\top \diag\left(\left[1-\tfrac{\partial h_1(x_1)}{\partial x_1},1-\tfrac{\partial h_{\vert\mathcal{V}\vert}(x_{\vert\mathcal{V}\vert})}{\partial x_{\vert\mathcal{V}\vert}}\right]\right)\proj_{S^\perp}(y)
    \end{aligned}
    \end{equation}
    where we use $Ph(x)=\proj_{S^\perp}(y)$.
    
    Then, we show that the spectral norm of the diagonal matrix $I-\tfrac{\partial h(x)}{\partial x}$ has an upper bound $M=\max(1,|1-m|)$. Recall that the spectral norm of a diagonal matrix is the largest absolute value of its diagonal entries. We have that $\|I-\tfrac{\partial h(x)}{\partial x}\|=\max_i\vert 1-\tfrac{\partial h_i(x_i)}{\partial x_i}\vert$. Since $h_i$ are monotone passive functions that have bounded derivatives, the relationship $0\leq\tfrac{\partial h_i(x_i)}{\partial x_i}\leq m$ holds for all $x_i$. Consequently, for all $i$, we have $1-m\leq 1-\tfrac{\partial h_i(x_i)}{\partial x_i}\leq 1$ and $\vert 1-\tfrac{\partial h_i(x_i)}{\partial x_i}\vert\leq M$.
    
    Now, apply the Cauchy-Schwarz inequality to \eqref{eq:uprojy_dotQ},
\begin{equation}
    \begin{aligned}
        u^\top (I-\tfrac{\partial h(x)}{\partial x})&\proj_{S^\perp}(y)\leq \vert u^\top (I-\tfrac{\partial h(x)}{\partial x})\proj_{S^\perp}(y)\vert \\
        &\leq \|I-\tfrac{\partial h(x)}{\partial x}\|_2 \|u\|_2^2\|\proj_{S^\perp}(y)\|_2^2 \\
        &\leq \tfrac{M}{2}\|u\|_2^2+\tfrac{M}{2}\|\proj_{S^\perp}(y)\|_2^2,
    \end{aligned}
\end{equation}
and we obtain that 
\begin{equation}\label{eq:lower_bound}
    \begin{aligned}
        u^\top \proj_{S^\perp}(y)&-\dot{Q}(x)=u^\top (I-\tfrac{\partial h(x)}{\partial x})\proj_{S^\perp}(y) \\
        &\geq -\tfrac{M}{2}\|u\|_2^2-\tfrac{M}{2}\|\proj_{S^\perp}(y)\|_2^2.
    \end{aligned}
\end{equation}
This completes the proof.
\end{proof}

From Proposition \ref{prop:IOP_controllers}, Theorem \ref{thm:SPassivity_outconsensus} and Proposition \ref{prop:dissipativity}, we can derive the sufficient conditions for $(\Upsilon,\Pi,\mathcal{D},w)$ to achieve output agreement.
\begin{corollary}\label{coro:passive_agents}
    Consider a network system $(\Upsilon,\Pi,\mathcal{D},w)$ where $\mathcal{D}$ is a digraph with a globally reachable node. Assume that the edge controllers are input-output passive. Let $\max(D_o)$ be the maximal out-degree of graph $\mathcal{D}$. For the agent relation $H_a$ with characterization \eqref{eq:upsilon_dissipativity}, if there exists a controller relation $H_c$ with characterization \eqref{eq:IOP_z_projy} where $\alpha\geq\max(D_o)\tfrac{M}{2}$ and $\gamma\lambda_2>\tfrac{M}{2}$, then, the network system achieves output agreement. \hfill $\triangledown$
\end{corollary}

This corollary highlights the interplay between the passivity properties of controllers and the structural properties of the graph (e.g., $\max(D_o)$ and $\lambda_2$) in shaping the output agreement behavior of networked systems $(\Upsilon,\Pi,\mathcal{D},w)$. Consequently, by carefully designing both the graph topology and the controllers, we can guarantee that the system achieves output agreement.

However, Proposition \ref{prop:dissipativity} and Corollary \ref{coro:passive_agents} exhibit certain limitations that warrant attention. First, the lower bound in \eqref{eq:lower_bound} holds for all $u^\top\proj_{S^\perp}(y)-\dot{Q}$, but a tighter bound on this difference remains to be determined. Additionally, the conditions provided are merely sufficient for achieving output consensus. Identifying a passivity-based condition that is both sufficient and necessary is left as an open problem for future research.

\section{CASE STUDIES}\label{sec:case_study}
In this section, we first leverage the analysis approach introduced in the previous section to provide a submanifold-based explanation of the consensus behaviors of linear consensus protocols. Subsequently, we analyze the output agreement problem for a heterogeneous network system.
\subsection{Linear consensus protocol for directed networks}
We recall that the linear consensus protocol for directed networks \cite{Mesbahi2010} discusses the consensus behavior of a group of integrators $\Upsilon_i^l$ that interact over a directed graph $\mathcal{G}$ with edge controllers $\Pi_k^l$, where
\begin{equation}
\begin{aligned}
    \Upsilon_i^l:\  \dot{x}_i=u_i, \ y_i=x_i,\ \
    \Pi_k^l:\ \mu_k=\zeta_k.
\end{aligned} 
\end{equation}
In this case, assume that the graph is directed with a globally reachable node. The protocol is denoted by $(\Upsilon^l,\Pi^l,\mathcal{D},w)$, and this system satisfies the relation given in \eqref{eq:net_relations}.

Integrators are passive \cite[Example 6.2]{Khalil2008Nonlinear}, and satisfy the conditions to apply Proposition \ref{prop:dissipativity}. Thus, to establish the $S$-constrained passivity, we choose the manifold-constained storage function as $Q(x)=\tfrac{1}{2}x^\top Px$ with $P=(I-\tfrac{1}{\vert\mathcal{V}\vert}\mathbb{1}\mathbb{1}^\top)$. Then, the agent relation $H_a$ satisfies,
\begin{align}
    u^\top \proj_{S^\perp}(y)\geq \dot{Q}(x).
\end{align}

Recall that the edge controllers $\Pi_k^l$ are output strictly passive (i.e., also input-output passive). Then, the passivity-based inequality of the controller relation $H_c$ satisfies,
\begin{equation}
\begin{aligned}
      z^\top \proj_{S\perp}(y)&=z^\top y=\mu^\top E^\top y=\mu^\top \zeta= \|\mu\|_2^2 \\&=\|\zeta\|_2^2
    =y^\top Ly\geq \lambda_2\|\proj_{S^\perp}(y)\|_2^2,
\end{aligned}
\end{equation}
where $\lambda_2>0$ because $\mathcal{D}$ contains a globally reachable node \cite[Theoreom 6.6, Corollart 6.8]{bullo2020lectures}. Then, by applying Corollary \ref{coro:passive_agents}, the linear consensus protocol $(\Upsilon^l,\Pi^l,\mathcal{D},w)$ achieves output agreement. This aligns with the consensus behavior of this protocol \cite{Mesbahi2010}.

\subsection{A heterogeneous network system}
Consider a networked system of $5$ SISO agents $\Upsilon_i$ with edge controllers $\Pi_k^l: \mu_k = b\zeta_k$. The output equations for $\Upsilon_i$ are $y_1 = x_1$, $y_2 = x_2$, $y_3 = \tanh(x_3)$, $y_4 = \tanh(x_4)$, and $y_5 = \tfrac{x_5}{(1 + |x_5|)}$. These output functions $h_i$'s are monotonically passive with derivatives upper-bounded by $1$, yielding $M = 1$ in \eqref{eq:upsilon_dissipativity} by Proposition \ref{prop:dissipativity}. The agents are interconnected over the strongly connected digraph shown in Fig. \ref{fig:cases_graph}. The graph has $\lambda_2=3$ and maximal out-degree $\max(D_o)=2$. Also, the corresponding controller relation $H_c$ satisfies,
\begin{equation*}
    z^\top\proj_{S^\perp}(y)=b\mu^\top \zeta= ab\|\zeta\|_2^2+\tfrac{1-a}{b}\|\mu\|_2^2, \ 0< a< 1.
\end{equation*}

By Corollary \ref{coro:passive_agents}, achieving output agreement requires $ab \geq 1$ and $\lambda_2 \tfrac{1-a}{b} > \tfrac{1}{2}$. Fig. \ref{fig:cases_agreement} shows illustrates the output trajectories of each agent with $b = 2$, $a = \tfrac{1}{2}$, and initial conditions $x(0) = [0.23,-0.2,1,-2.4,0]^\top$. Since both conditions are satisfied, the heterogeneous networked system achieves output consensus with agreement value $0.1776$.

\begin{figure}[!ht]
\centering
    \begin{subfigure}{0.10\textwidth}
\centering
\raisebox{0.8cm}{
    \begin{tikzpicture}[scale = .6]
        \Vertex[size=.3,x=0, y=0, label=$1$,opacity =.5]{v1} 
        \Vertex[size=.3,x=-1.5, y=-1.3, label=$2$,opacity =.5]{v2} 
        \Vertex[size=.3,x=1.5, y=-1.3, label=$3$,opacity =.5]{v3}
        \Vertex[size=.3,x=-1, y=-2.6, label=$4$
        ,opacity =.5]{v4}
        \Vertex[size=.3,x=1, y=-2.6, label=$5$,opacity =.5]{v5}
       
        \Edge[Direct](v1)(v2)
        \Edge[Direct](v2)(v3)
        \Edge[Direct](v4)(v3)
        \Edge[Direct](v5)(v4)
        \Edge[Direct](v2)(v5)
        \Edge[Direct](v3)(v1)
        \Edge[Direct](v4)(v1)
        \Edge[Direct](v5)(v1)

       \end{tikzpicture}
       }
       \centering
    \caption[]{}
    \label{fig:cases_graph}
\end{subfigure}
\begin{subfigure}{0.35\textwidth}
    \centering
    \includegraphics[width=0.65\linewidth]{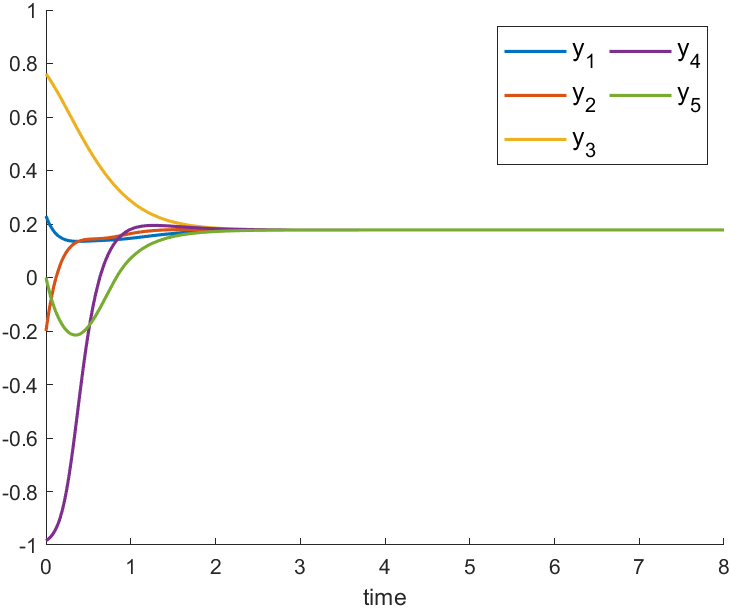}
    \caption[]{}
    \label{fig:cases_agreement}
\end{subfigure}
\caption{The output agreement for the heterogeneous network. (a) The underlying digraph. (b) The output of each agent.}
\end{figure}

\begin{figure}[!h]
\centering
\begin{subfigure}{0.10\textwidth}
\centering
\raisebox{0.8cm}{
    \begin{tikzpicture}[scale = .6]
        
        \Vertex[size=.3,x=0, y=0, label=$1$,opacity =.5]{v1} 
        \Vertex[size=.3,x=-0.8, y=-1, label=$2$,opacity =.5]{v2} 
        \Vertex[size=.3,x=-1.6, y=-2, label=$3$,opacity =.5]{v3}
        \Vertex[size=.3,x=0.8, y=-1, label=$4$
        ,opacity =.5]{v4}
        \Vertex[size=.3,x=1.6, y=-2, label=$5$,opacity =.5]{v5}
       
        \Edge[Direct](v3)(v2)
        \Edge[Direct](v2)(v1)
        \Edge[Direct](v5)(v4)
        \Edge[Direct](v4)(v1)

       \end{tikzpicture}
       }
       \centering
    \caption[]{}
    \label{fig:cases_unbgraph}
\end{subfigure}
\begin{subfigure}{0.35\textwidth}
    \centering
    \includegraphics[width=0.65\linewidth]{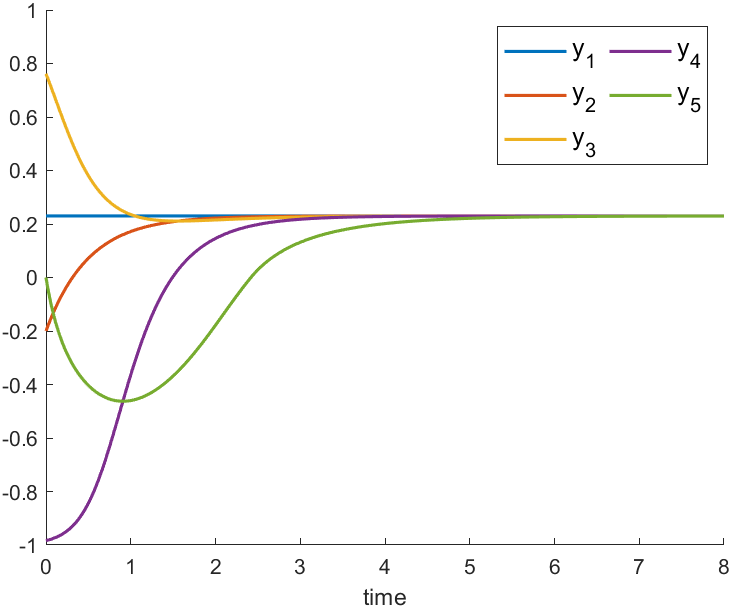}
    \caption[]{}
    \label{fig:cases_negative}
\end{subfigure}
\caption{A negative example of the output agreement for the heterogeneous network. (a) The underlying digraph. (b) The output of each agent.}
\end{figure}

We now present an example to demonstrate that Corollary \ref{coro:passive_agents} provides only sufficient conditions for output agreement. Consider the same agents and controllers interconnected over the digraph shown in Fig. \ref{fig:cases_unbgraph}, where $\lambda_2 = 0.382$ and $\max(D_o) = 1$. For this graph, although no values of $a$ and $b$ satisfy both $ab \geq \tfrac{1}{2}$ and $\lambda_2 \tfrac{1-a}{b} > \tfrac{1}{2}$, the network still achieves output consensus with agreement value $0.23$.

\section{CONCLUDING REMARKS}\label{sec:concluding}
In this work, we present a passivity-based analysis for the output agreement problem of network systems over directed graphs. First, we establish connections between passivity properties of system relations and the passivity of individual agents (controllers), revealing fundamental limitations. To address these limitations, we introduce submanifold-constrained passivity and develop a compensation theorem that guarantees output agreement. Finally, we apply the new theorem to analyze network systems composed of passive agents. In future work, we aim to establish passivity-based sufficient and necessary conditions for network systems to achieve output agreement, and extend the $S$-passivity (Definition \ref{def:SPassivity}) to arbitrary embedded submanifold.

\bibliographystyle{IEEEtran}
\bibliography{ref}




\section*{APPENDIX}
\subsection*{Proof of Proposition \ref{prop:IOP_agents}}
According to \cite[Proposition 5]{yue2024balance_passivity}, we have $u^\top \proj_{S^\perp}(y)\geq \sum_{i=1}^{|\mathcal{V}|} \dot{V}_i+\delta\|u\|_2^2+\varepsilon\|y\|_2^2-\|u\|_2\|y\|_2$. Apply Young's inequality (i.e., $ab\leq\tfrac{a^2+b^2}{2}$ for nonnegative $a,b$) to $\|u\|_2\|y\|_2$, and we complete the proof.

\subsection*{Proof of Proposition \ref{prop:IOP_controllers}}
To establish the relationship between controllers' passivity and the characterization of the controller relation $H_c$, we need the following statement.
\begin{proposition}\label{prop:yLy_projy}
    Consider a digraph $\mathcal{D}$ with a globally reachable node. Let $\mathbb{G}=(\mathcal{V},\mathbb{E})$ be the undirected counterpart of $\mathcal{D}$ and $L(\mathbb{G})$ be the graph Laplacian of graph $\mathbb{G}$. Then, for any vectors $y\in\mathbb{R}^{{\vert\mathcal{V}\vert}}$, the following inequality holds,
    \begin{equation}
        \lambda_2 \|\proj_{S^\perp}(y)\|_2^2  \leq y^\top L(\mathbb G)y\leq \lambda_{\vert\mathcal{V}\vert}  \|\proj_{S^\perp}(y)\|_2^2,
    \end{equation}
    where $\lambda_2$ and $\lambda_{\vert\mathcal{V}\vert}$ denote the second smallest and the largest eigenvalues of $L(\mathbb{G})$, respectively. \hfill $\triangledown$
\end{proposition}

\begin{proof}
    According to the given conditions, $\mathbb{G}$ is a connected graph. So, let $0=\lambda_1<\lambda_2\leq\ldots\leq\lambda_{\vert\mathcal{V}\vert}$ be the eigenvalues of $L(\mathbb{G})$ and $\{v_1,v_2,\ldots,v_{\vert\mathcal{V}\vert}\}$ be the corresponding orthogonal set of eigenvectors. Note that these eigenvectors also form an orthogonal basis of $\mathbb{R}^{{\vert\mathcal{V}\vert}}$ where 
    the vector $v_1=\tfrac{1}{\sqrt n}\,{\bf 1}_n$ spans the \emph{agreement subspace} $S:=\operatorname{span}\{{\bf 1}_n\}$, while  $\{v_2,\ldots,v_{n}\}$ is an orthonormal basis for the \emph{disagreement subspace} $S^{\perp}$.  
    Because $L(\mathbb G)$ is symmetric, it has the orthogonal decomposition, i.e., $L(\mathbb G)=\sum\limits_{i=1}^{\vert\mathcal{V}\vert}\lambda_i v_iv_i^\top$. 

    Any vectors $y\in\mathbb{R}^{{\vert\mathcal{V}\vert}}$, can likewise be decomposed as $y=\tfrac{1}{\vert\mathcal{V}\vert}\mathbb{1}\mathbb{1}^\top y+m$, where $m=\proj_{S^\perp}(y)$ and $y^\top(I-\tfrac{1}{\vert\mathcal{V}\vert}\mathbb{1}\mathbb{1}^\top)y=m^\top m$. Also, since $L(\mathbb G)\mathbb{1}_{\vert\mathcal{V}\vert}=\mathbb{0}_{\vert\mathcal{V}\vert}$, it follows that $y^\top L(\mathbb G)y=m^\top L(\mathbb G)m$. Note that the vector $m$ can also be represented by the eigenbasis, i.e., $m=\sum\limits_{i=2}^{\vert\mathcal{V}\vert}b_iv_i$.

    Then, we have 
    \begin{align*}
        y^\top L(\mathbb G) y&=m^\top L(\mathbb G)m=\left(\sum\limits_{i=2}^{\vert\mathcal{V}\vert}b_iv_i\right)^\top L (\mathbb G)\left(\sum\limits_{i=2}^{\vert\mathcal{V}\vert}b_iv_i\right)\\
        &\hspace{-1cm}=\sum\limits_{i=2}^{\vert\mathcal{V}\vert}\sum\limits_{j=2}^{\vert\mathcal{V}\vert}b_ib_jv_i^\top L(\mathbb G)v_j
        =\sum\limits_{i=2}^{\vert\mathcal{V}\vert}b_i^2\lambda_iv_i^\top v_i=\sum\limits_{i=2}^{\vert\mathcal{V}\vert}\lambda_ib_i^2,
    \end{align*}
where we use $L(\mathbb G)v_j=\lambda_jv_j$, $v_i^\top v_j=0$ and $v_i^\top v_i=1$. Apply Rayleigh quotient,
\begin{equation*}
    \tfrac{y^\top L(\mathbb G) y}{\proj_{S^\perp}^\top(y)\proj_{S^\perp}(y)}=\tfrac{m^\top L(\mathbb G) m}{m^\top m}=\tfrac{\sum_{i=2}^{\vert\mathcal{V}\vert}\lambda_ib_i^2}{\sum_{i=2}^{\vert\mathcal{V}\vert}b_i^2},
\end{equation*}
where we arrive at our desired result.
\end{proof}

Now, we are ready to prove Proposition \ref{prop:IOP_controllers}.

Let $\mathbb{G}$ and $L(\mathbb{G})$ be defined as above, and $E$ be the incidence matrix of the graph $\mathcal{D}$. Sum up the passivity inequalities of all the controllers,
    \begin{equation}
    \begin{aligned}\label{eq:controller_passivity}
        z^\top \proj_{S^\perp}(y)&= \mu^\top E^\top \proj_{S^\perp}(y)\\
        &=\mu^\top E^\top y-\tfrac{1}{\vert \mathcal{V} \vert}\mu^\top E^\top \mathbb{1}_{\vert \mathcal{V} \vert}\mathbb{1}_{\vert \mathcal{V} \vert}^\top y\\
        &=\mu^\top E^\top y=\mu^\top\zeta \geq \sum_{k=1}^{|\mathcal{E}|} \dot{W}_k+\sum_{k=1}^{|\mathcal{E}|} \alpha_k \mu_k^2 \\ &+\sum_{k=1}^{|\mathcal{E}|} \gamma_k \zeta_k^2 
         \geq \sum_{k=1}^{|\mathcal{E}|} \dot{W}_k+\alpha\|\mu\|_2^2+\gamma\|\zeta\|_2^2 \\
         & = \sum_{k=1}^{|\mathcal{E}|} \dot{W}_k+\alpha\|\mu\|_2^2+\gamma y^\top L(\mathbb{G}) y,
         \end{aligned}
    \end{equation}
    where we use the properties of incidence matrix, i.e., $E^\top\mathbb{1}_{\vert \mathcal{V} \vert}=\mathbb{0}_{\vert \mathcal{E} \vert}$ and $E E^\top=L(\mathbb{G})$.

\subsection*{Proof of Corollary \ref{coro:passive_agents}}
     We start by showing that the trajectories of $(\Upsilon,\Pi,\mathcal{G},w)$ are bounded. First, the forward path of Figure \ref{fig:decomp} (from $u$ to $y$) is passive because of the passive agents $\Upsilon_i$. And the storage function for $\Upsilon$ can be chosen as $\sum_{i=1}^{\vert\mathcal{V}\vert}\int_0^{x_i}h_i(s){\rm d}s$, which is radially unbounded. Under the given assumptions, the feedback path of Figure \ref{fig:decomp} (from $y$ to $z$) is also passive. Consequently, it satisfies a global dissipation inequality, where the rate of change of the storage function is bounded by the supply rate. The storage function can be chosen to be radially unbounded, so the trajectories must be bounded. It follows that $u(t),y(t),\zeta(t)$ and $\mu(t)$ are bounded. Thus, the agent relation $H_a$ and the controller relations $H_c$ are bounded.

     For inequality \eqref{eq:upsilon_dissipativity}, since $u=-B_o\mu$ and $-\|u\|_2^2\geq -\max(D_o)\|\mu\|_2^2$ \cite{yue2024balance_passivity}, it follows that,
     \begin{small}
         \begin{equation}\label{eq:SP_agents}
         u^\top \proj_{S^\perp}(y)\geq \dot{Q}(x)-\tfrac{\max(D_o)M}{2}\|\mu\|_2^2-\tfrac{M}{2}\|\proj_{S^\perp}(y)\|_2^2.
     \end{equation}
     \end{small}
     
     Then, the controller relation $H_c$ with $\alpha\geq\max(D_o)\tfrac{M}{2}$ and $\gamma\lambda_2> \tfrac{M}{2}$ compensate for the negativity terms in \eqref{eq:SP_agents}. Now, we have established the conditions needed to apply Theorem \ref{thm:SPassivity_outconsensus}. This completes our proof.


\end{document}